\documentclass[11pt]{article}

\usepackage{authblk,amsmath,amsfonts,amssymb,dsfont,fullpage,hyperref}
\usepackage[amsmath,thmmarks,thref]{ntheorem}
\usepackage[round]{natbib}
\def\E{\mathbb{E}}
\def\R{\mathbb{R}}
\def\Sig{\varSigma}
\def\I{\mathds{1}}
\def\k{\bar{k}}
\DeclareMathOperator{\diag}{diag}
\DeclareMathOperator{\rank}{rank}
\DeclareMathOperator{\tr}{tr}

\def\qed{\ensuremath{\square}}
\makeatletter
\newtheoremstyle{plain-theorem}%
{\item[\hskip\labelsep \theorem@headerfont ##1\ ##2\theorem@separator]}%
{\item[\hskip\labelsep \theorem@headerfont ##1\ ##2\ {\normalfont({##3})}\theorem@separator]}%
\newtheoremstyle{nonumberplain-proof}%
{\item[\hskip\labelsep \theorem@headerfont ##1\theorem@separator]}%
{\item[\hskip\labelsep \theorem@headerfont ##3\theorem@separator]}%
\makeatother
\theoremnumbering{arabic}
\theoremstyle{plain-theorem}
\theorembodyfont{\itshape}
\theoremheaderfont{\normalfont\bfseries}
\theoremseparator{.}
\newtheorem{theorem}{Theorem}
\newtheorem{lemma}{Lemma}

\theorembodyfont{\normalfont}
\theoremsymbol{\qed}
\newtheorem{example}{Example}
\newtheorem{remark}{Remark}
\theoremstyle{nonumberplain-proof}
\theorembodyfont{\normalfont}
\theoremheaderfont{\itshape}
\theoremsymbol{\qed}
\newtheorem{proof}{Proof}

\title{Dimension-free tail inequalities for sums of random matrices
}
\author[1,2]{Daniel Hsu}
\author[2]{Sham M.~Kakade}
\author[1]{Tong Zhang}
\affil[1]{Department of Statistics, Rutgers University}
\affil[2]{Department of Statistics, Wharton School, University of Pennsylvania}

\begin{document}
\maketitle
{\def\thefootnote{}
\footnotetext{E-mail: \texttt{djhsu@rci.rutgers.edu},
\texttt{skakade@wharton.upenn.edu}, \texttt{tzhang@stat.rutgers.edu}}}

\begin{abstract}
We derive exponential tail inequalities for sums of random matrices with no
dependence on the explicit matrix dimensions.
These are similar to the matrix versions of the Chernoff bound and
Bernstein inequality except with the explicit matrix dimensions replaced by
a trace quantity that can be small even when the dimension is large or
infinite.
Some applications to principal component analysis and approximate matrix
multiplication are given to illustrate the utility of the new bounds.
\end{abstract}

\section{Introduction}

Sums of random matrices arise in many statistical and probabilistic
applications, and hence their concentration behavior is of significant
interest.
Surprisingly, the classical exponential moment method used to derive tail
inequalities for scalar random variables carries over to the matrix setting
when augmented with certain matrix trace inequalities.
This fact was first discovered by~\citet{AhlWin02}, who proved a matrix
version of the Chernoff bound using the Golden-Thompson
inequality~\citep{Golden65,Thompson65}: $\tr \exp(A+B) \leq
\tr(\exp(A)\exp(B))$ for all symmetric matrices $A$ and $B$.
Later, it was demonstrated that the same technique could be adapted
to yield analogues of other tail bounds such as Bernstein's
inequality~\citep{GLFBE10,Recht09,Gross09,Oli10a,Oli10b}.
Recently, a theorem due to~\citet{Lieb73} was identified
by~\citet{Tropp11a,Tropp11b} to yield sharper versions of this general
class of tail bounds.
Altogether, these results have proved invaluable in constructing and
simplifying many probabilistic arguments concerning sums of random
matrices.

One deficiency of these previous inequalities is their explicit dependence
on the dimension, which prevents their application to infinite dimensional
spaces that arise in a variety of data analysis
tasks~\citep[\emph{e.g.},][]{SSM99,RasWil06,FBG07,Bach08}.
In this work, we prove analogous results where the dimension is replaced
with a trace quantity that can be small even when the dimension is large or
infinite.
For instance, in our matrix generalization of Bernstein's inequality, the
(normalized) trace of the second moment matrix appears instead of the
matrix dimension.
Such trace quantities can often be regarded as an intrinsic notion of
dimension.
The price for this improvement is that the more typical exponential tail
$e^{-t}$ is replaced with a slightly weaker tail $t (e^t-t-1)^{-1} \approx
e^{-t + \log t}$.
As $t$ becomes large, the difference becomes negligible.
For instance, if $t \geq 2.6$, then $t (e^t-t-1)^{-1} \leq e^{-t/2}$.

There are some previous works that give dimension-free tail inequalities in
some special cases.
\citet{RV07} prove exponential tail inequalities for sums of rank-one
matrices by way of a key inequality of~\citet{Rudelson99} \citep[see
also][]{Oli10a}.
\citet{MagZou11} prove tail inequalities for sums of low-rank matrices using
non-commutative Khintchine moment inequalities, but fall short of giving an
exponential tail inequality.
In contrast, our results are proved using a natural matrix generalization
of the exponential moment method.

\section{Preliminaries}

Let $\xi_1,\dotsc,\xi_n$ be random variables, and for each $i=1,\dotsc,n$,
let $X_i := X_i(\xi_1,\dotsc,\xi_i)$ be a symmetric matrix-valued
functional of $\xi_1,\dotsc,\xi_i$.
We use $\E_i[ \ \cdot \ ]$ and shorthand for $\E[ \ \cdot \ | \
\xi_1,\dotsc,\xi_{i-1} ]$.
For any symmetric matrix $H$, let $\lambda_{\max}(H)$ denote its largest
eigenvalue, $\exp(H) := I + \sum_{k=1}^\infty H^k / k!$, and $\log(\exp(H))
:= H$.

The following convex trace inequality of~\citet{Lieb73} was also used
by~\citet{Tropp11a,Tropp11b}.
\begin{theorem}[\citealp*{Lieb73}] \label{theorem:lieb}
For any symmetric matrix $H$, the function $M \mapsto \tr\exp(H + \log(M))$
is concave in $M$ for $M \succ 0$.
\end{theorem}

The following lemma due to \citep{Tropp11b} is a matrix generalization of a
scalar result due to~\citet{Freedman75} \citep[see also][]{Zhang05a}, where
the key is the invocation of Theorem~\ref{theorem:lieb}.
We give the proof for completeness.
\begin{lemma}[\citealp*{Tropp11b}] \label{lemma:laplace}
For any constant symmetric matrix $X_0$,
\begin{equation} \label{eq:laplace}
\E\left[ \tr\exp\left( \sum_{i=0}^n X_i
- \sum_{i=1}^n \ln \E_i\left[\exp(X_i)\right] \right) \right]
\leq \tr\exp(X_0)
.
\end{equation}
\end{lemma}
\begin{proof}
By induction on $n$.
The claim holds trivially for $n = 0$.
Now fix $n \geq 1$, and assume as the inductive hypothesis
that~\eqref{eq:laplace} holds with $n$ replaced by $n-1$.
In this case,
\begin{align*}
\E\left[ \tr\exp\left( \sum_{i=0}^n X_i - \sum_{i=1}^n \log
\E_i\left[\exp(X_i)\right] \right)\right]
& = \E\left[\E_n\left[ \tr\exp\left( \sum_{i=0}^{n-1} X_i - \sum_{i=1}^n
\log \E_i\left[\exp(X_i)\right] + \log \exp(X_n) \right)\right]\right]
\\
& \leq \E\left[\tr\exp\left( \sum_{i=0}^{n-1} X_i -
\sum_{i=1}^n \log \E_i\left[\exp(X_i)\right]
+ \log\E_n\left[\exp(X_n)\right] \right)\right]
\\
& = \E\left[\tr\exp\left( \sum_{i=0}^{n-1} X_i -
\sum_{i=1}^{n-1} \log \E_i\left[\exp(X_i)\right] \right)\right]
\\
& \leq \tr\exp(X_0)
\end{align*}
where the first inequality follows from Theorem~\ref{theorem:lieb} and
Jensen's inequality, and the second inequality follows from the inductive
hypothesis.
\end{proof}

\section{Exponential tail inequalities for sums of random matrices}

\subsection{A generic inequality}

We first state a generic inequality based on Lemma~\ref{lemma:laplace}.
This differs from earlier approaches, which instead combine Markov's
inequality with a result similar to
Lemma~\ref{lemma:laplace}~\citep[\emph{e.g.},][Theorem 3.6]{Tropp11a}.
\begin{theorem} \label{theorem:laplace-prob}
For any $\eta \in \R$ and any $t > 0$,
\begin{equation*}
\Pr\left[ \lambda_{\max}\left(
\eta \sum_{i=1}^n X_i - \sum_{i=1}^n \log \E_i\left[\exp(\eta X_i)\right]
\right) > t \right]
\leq
\tr\left(\E\left[ - \eta \sum_{i=1}^n X_i +
\sum_{i=1}^n \log
\E_i\left[ \exp(\eta X_i) \right] \right]
\right)
\cdot (e^t-t-1)^{-1}
.
\end{equation*}
\end{theorem}
\begin{proof}
Fix a constant matrix $X_0$, and let $A := \eta \sum_{i=0}^n X_i -
\sum_{i=1}^n \log \E_i[\exp(\eta X_i)]$.
Note that $g(x) := e^x - x - 1$ is non-negative for all $x \in \R$ and
increasing for $x \geq 0$.
Letting $\{ \lambda_i(A) \}$ denote the eigenvalues of $A$, we have
\begin{align*}
\Pr\left[ \lambda_{\max}(A) > t \right] (e^t - t - 1)
& = \E\left[ \I\bigl[ \lambda_{\max}(A) > t \bigr] (e^t - t - 1) \right]
\\
& \leq \E\left[ e^{\lambda_{\max}(A)} - \lambda_{\max}(A) - 1 \right] \\
& \leq \E\left[ \sum_i \bigl( e^{\lambda_i(A)} - \lambda_i(A) - 1 \bigr)
\right] \\
& = \E\left[ \tr(\exp(A) - A - I) \right] \\
& \leq \tr(\exp(X_0) + \E[-A] - I)
\end{align*}
where the last inequality follows from Lemma~\ref{lemma:laplace}.
Now we take $X_0 \to 0$ so $\tr(\exp(X_0) - I) \to 0$.
\end{proof}

%
\subsection{Some specific bounds}

We now give some specific bounds as corollaries of
Theorem~\ref{theorem:laplace-prob}.
Most of the estimates used in the proofs are taken from previous
works~\citep[\emph{e.g.},][]{AhlWin02,Tropp11a}; the main point here is to
show how these previous techniques can be combined with
Theorem~\ref{theorem:laplace-prob} to yield new tail inequalities with no
explicit dependence on the matrix dimension.

First, we give a bound under a subgaussian-type condition on the
distribution.
\begin{theorem}[Matrix subgaussian bound] \label{theorem:subgaussian}
If there exists $\bar\sigma > 0$ and $\k > 0$ such that
for all $i=1,\dotsc,n$,
\begin{gather*}
\E_i[X_i] = 0 \\
\lambda_{\max}\left(
\frac1n
\sum_{i=1}^n \log \E_i\bigl[ \exp(\eta X_i) \bigr]
\right)
\leq \frac{\eta^2\bar\sigma^2}{2} \\
\E\left[ \tr\left(
\frac1n
\sum_{i=1}^n \log \E_i\bigl[ \exp(\eta X_i) \bigr]
\right)
\right]
\leq \frac{\eta^2\bar\sigma^2\k}{2}
\end{gather*}
for all $\eta > 0$ almost surely, then for any $t > 0$,
\begin{equation*}
\Pr\left[ \lambda_{\max}\left(\frac1n\sum_{i=1}^n X_i\right) >
\sqrt{\frac{2\bar\sigma^2 t}{n}} \right]
\leq \k \cdot t(e^t - t - 1)^{-1}
.
\end{equation*}
\end{theorem}
\begin{proof}
We fix $\eta := \sqrt{2t/(\bar\sigma^2 n)}$.
By Theorem~\ref{theorem:laplace-prob}, we obtain
\begin{align*}
\Pr\left[ \lambda_{\max}\left(
\frac1n\sum_{i=1}^n X_i
- \frac1{n\eta} \sum_{i=1}^n \log \E_i\left[\exp(\eta X_i)\right]
\right) > \frac{t}{n\eta} \right]
& \leq
\tr\left(\E\left[
\sum_{i=1}^n \log
\E_i\left[ \exp(\eta X_i) \right] \right]
\right)
\cdot (e^t-t-1)^{-1}
\\
& \leq
\frac{n\eta^2\bar\sigma^2\k}{2}
\cdot (e^t-t-1)^{-1}
\\
& = \k \cdot t(e^t-t-1)^{-1}
.
\end{align*}
Now suppose
\begin{equation*}
\lambda_{\max}\left(
\frac1n\sum_{i=1}^n X_i
- \frac1{n\eta} \sum_{i=1}^n \log \E_i\left[\exp(\eta X_i)\right]
\right) \leq \frac{t}{n\eta}
.
\end{equation*}
This implies for every non-zero vector $u$,
\begin{equation*}
\frac{u^\top\left( \frac1n\sum_{i=1}^n X_i \right) u}{u^\top u}
\leq \frac{u^\top\left(
\frac1{n\eta} \sum_{i=1}^n \log \E_i\left[\exp(\eta X_i)\right]
\right) u}{u^\top u}
+ \frac{t}{n\eta}
\leq
\lambda_{\max}\left(
\frac1{n\eta} \sum_{i=1}^n \log \E_i\left[\exp(\eta X_i)\right]
\right)
+ \frac{t}{n\eta}
\end{equation*}
and therefore
\begin{equation*}
\lambda_{\max}\left( \frac1n\sum_{i=1}^n X_i \right)
\leq
\lambda_{\max}\left(
\frac1{n\eta} \sum_{i=1}^n \log \E_i\left[\exp(\eta X_i)\right]
\right)
+ \frac{t}{n\eta}
\leq \frac{\eta\bar\sigma^2}{2} + \frac{t}{n\eta}
= \sqrt{\frac{2\bar\sigma^2t}{n}}
\end{equation*}
as required.
\end{proof}

We can also give a Bernstein-type bound based on moment conditions.
For simplicity, we just state the bound in the case that the
$\lambda_{\max}(X_i)$ are bounded almost surely.
\begin{theorem}[Matrix Bernstein bound] \label{theorem:bernstein}
If there exists $\bar b > 0$, $\bar\sigma > 0$, and $\k > 0$ such that
for all $i=1,\dotsc,n$,
\begin{gather*}
\E_i[X_i] = 0 \\
\lambda_{\max}(X_i) \leq \bar b \\
\lambda_{\max}\left( \frac1n \sum_{i=1}^n \E_i[X_i^2] \right)
\leq \bar\sigma^2 \\
\E\left[ \tr\left(\frac1n\sum_{i=1}^n \E_i[X_i^2] \right) \right]
\leq \bar\sigma^2\k
\end{gather*}
almost surely, then for any $t > 0$,
\begin{equation*}
\Pr\left[
\lambda_{\max}\left( \frac1n \sum_{i=1}^n X_i \right)
> \sqrt{\frac{2\bar\sigma^2t}{n}} + \frac{\bar bt}{3n}
\right]
\leq \k \cdot t (e^t - t - 1)^{-1}
.
\end{equation*}
\end{theorem}
\begin{proof}
Let $\eta > 0$.
For each $i=1,\dotsc,n$,
\begin{equation*}
\exp(\eta X_i) \preceq I + \eta X_i + \frac{e^{\eta\bar b} - \eta\bar b -
1}{\bar b^2} \cdot X_i^2
\end{equation*}
and therefore
\begin{equation*}
\log\E_i\bigl[\exp(\eta X_i)\bigr]
\preceq \frac{e^{\eta\bar b} - \eta\bar b -
1}{\bar b^2} \cdot \E_i\bigl[X_i^2\bigr]
.
\end{equation*}
Since $e^x-x-1 \leq x^2/(2(1-x/3))$ for $0 \leq x < 3$, we have by
Theorem~\ref{theorem:laplace-prob}
\begin{equation*}
\Pr\left[
\lambda_{\max}\left( \frac1n \sum_{i=1}^n X_i \right)
> \frac{\eta\bar\sigma^2}{2(1 - \eta\bar b/3)} + \frac{t}{\eta n}
\right]
\leq \frac{\eta^2\bar\sigma^2\k n}{2(1-\eta\bar b/3)}
\cdot (e^t - t - 1)^{-1}
\end{equation*}
provided that $\eta < 3/\bar b$.
Choosing
\begin{equation*}
\eta := \frac{3}{\bar b} \cdot
\left( 1 - \frac{\sqrt{2\bar\sigma^2t/n}}
{2\bar bt/(3n) + \sqrt{2\bar\sigma^2t/n}} \right)
\end{equation*}
gives the desired bound.
\end{proof}

\subsection{Discussion} \label{section:discussion}

The advantage of our results here over previous exponential tail
inequalities for sums of random matrices is the absence of explicit
dependence on the matrix dimensions.
Indeed, all previous tail inequalities using the exponential moment method
(either via the Golden-Thompson inequality or Lieb's trace inequality) are
roughly of the form $d \cdot e^{-t}$ when the matrices in the sum are $d
\times d$~\citep{AhlWin02,GLFBE10,Recht09,Gross09,Tropp11a,Tropp11b}.
Our results also improve over the tail inequalities of~\citet{RV07} in that
it applies to full-rank matrices, not just rank-one matrices; and also over
that of~\citet{MagZou11} in that it provides an exponential tail
inequality, rather than just a polynomial tail.
Thus, our improvements widen the applicability of these inequalities (and
the matrix exponential moment method in general); we explore some of these
in Subsection~\ref{section:examples}.

One disadvantage of our technique is that in finite dimensional settings,
the relevant trace quantity that replaces the dimension may turn out to be
of the same order as the dimension $d$ (an example of such a case is
discussed next).
In such cases, the resulting tail bound from
Theorem~\ref{theorem:bernstein} (say) of $\bar{k} \cdot t(e^t-t-1)^{-1}$ is
looser than the $d \cdot e^{-t}$ tail bound provided by earlier
techniques~\citep[\emph{e.g.},][]{Tropp11a}.

We note that the matrix exponential moment method used here and in previous
work can lead to a significantly suboptimal tail inequality in some cases.
This was pointed out by~\citet[Section 4.6]{Tropp11a}, but we elaborate on
it here further.
Suppose $x_1,\dotsc,x_n \in \{\pm1\}^d$ are i.i.d.~random vectors with
independent Rademacher entries---each coordinate of $x_i$ is $+1$ or $-1$
with equal probability.
Let $X_i = x_ix_i^\top - I$, so $\E[X_i] = 0$, $\lambda_{\max}(X_i) =
\lambda_{\max}(\E[X_i^2]) = d - 1$, and $\tr(\E[X_i^2]) = d(d-1)$.
In this case, Theorem~\ref{theorem:bernstein} implies the bound
\[ \Pr\left[ \lambda_{\max}\left( \frac1n \sum_{i=1}^n x_ix_i^\top - I
\right) > \sqrt{\frac{2(d-1)t}{n}} + \frac{(d-1)t}{3n} \right]
\leq dt(e^t-t-1)^{-1}
.
\]
On the other hand, because the $x_i$ have subgaussian projections, it is
known that
\[
\Pr\left[ \lambda_{\max}\left( \frac1n \sum_{i=1}^n x_ix_i^\top - I \right)
> 2\sqrt{\frac{71d+16t}{n}} + \frac{10d + 2t}{n} \right] \leq 2e^{-t/2}
\]
\citep[][also see Lemma~\ref{lemma:covariance-subgaussian} in
Appendix~\ref{appendix:outer}]{LPRTJ05}.
First, this latter inequality removes the $d$ factor on the right-hand
side.
Perhaps more importantly, the deviation term $t$ does not scale with $d$ in
this inequality, whereas it does in the former.
Thus this latter bound provides a much stronger exponential tail: roughly
put, $\Pr[\lambda_{\max}(\sum_{i=1}^n x_ix_i^\top /n - I) > c \cdot
(\sqrt{d/n} + d/n) + \tau] \leq \exp(-\Omega(n\min(\tau,\tau^2)))$ for some
constant $c > 0$; the probability bound from
Theorem~\ref{theorem:bernstein} is only of the form
$\exp(-\Omega((n/d)\min(\tau,\tau^2)))$.
The sub-optimality of Theorem~\ref{theorem:bernstein} is shared by all other
existing tail inequalities proved using this exponential moment method.
The issue is related to the asymptotic freeness of the random matrices
$X_1,\dotsc,X_n$~\citep{Voi91,Gui04}---\emph{i.e.}, that nearly all
high-order moments of random matrices vanish asymptotically---which is not
exploited in the matrix exponential moment method.
This means that the proof technique in the exponential moment method
over-counts the contribution of high-order matrix moments that should have
vanished.  
Formalizing this discrepancy would help clarify the limits of this
technique, but the task is beyond the scope of this paper. 
It is also worth mentioning that asymptotic freeness only holds when the
$X_i$ have independent entries.
For matrices with correlated entries, our bound is close to best possible
in the worst case.

\subsection{Examples} \label{section:examples}

For a matrix $M$, let $\|M\|_F$ denote its Frobenius norm, and let
$\|M\|_2$ denote its spectral norm.
If $M$ is symmetric, then $\|M\|_2 = \max\{\lambda_{\max}(M), \
-\lambda_{\min}(M)\}$, where $\lambda_{\max}(M)$ and $\lambda_{\min}(M)$
are, respectively, the largest and smallest eigenvalues of $M$.

\subsubsection{Supremum of a random process}

The first example embeds a random process in a diagonal matrix to show that
Theorem~\ref{theorem:subgaussian} is tight in certain cases.
\begin{example} \label{example:sup}
Let $(Z_1,Z_2,\dotsc)$ be (possibly dependent) mean-zero subgaussian random
variables; \emph{i.e.}, each $\E[Z_i] = 0$, and there exists positive
constants $\sigma_1,\sigma_2,\dotsc$ such that
\[ \E[\exp(\eta Z_i)] \leq \exp\left(\frac{\eta^2\sigma_i^2}{2}\right)
\quad \forall \eta \in \R
. \]
We further assume that $v := \sup_i \{ \sigma_i^2 \} < \infty$ and $k :=
\frac1v \sum_i \sigma_i^2 < \infty$.
Also, for convenience, we assume $\log k \geq 1.3$ (to simplify the tail
inequality).

Let $X = \diag(Z_1,Z_2,\dotsc)$ be the random diagonal matrix with the
$Z_i$ on its diagonal.
We have $\E[X] = 0$, and
\[
\log \E[\exp(\eta X)] \preceq \diag\left(\frac{\eta^2\sigma_1^2}{2},
\frac{\eta^2\sigma_2^2}{2}, \dotsc \right)
,
\]
so
\[
\lambda_{\max}\left(\log\E[\exp(\eta X)]\right)
\leq \frac{\eta^2 v}{2}
\quad \text{and} \quad
\tr\left(\log\E[\exp(\eta X)]\right)
\leq \frac{\eta^2 vk}{2}
.
\]
By Theorem~\ref{theorem:subgaussian}, we have
\[
\Pr\left[ \lambda_{\max}(X) > \sqrt{2vt} \right] \leq kt(e^t-t-1)^{-1}
.
\]
Therefore, letting $t := 2(\tau + \log k) > 2.6$ for $\tau > 0$ and
interpreting $\lambda_{\max}(X)$ as $\sup_i \{ Z_i \}$,
\[
\Pr\left[ \sup_i\{ Z_i \} > 2\sqrt{\sup_i \{ \sigma_i^2 \} \left(\log
\frac{\sum_i \sigma_i^2}{\sup_i \{ \sigma_i^2 \}} + \tau \right) } \right]
\leq e^{-\tau}
.
\]

Suppose the $Z_i \sim \mathcal{N}(0,1)$ are just $N$ i.i.d.~standard
Gaussian random variables.
Then the above inequality states that the largest of the $Z_i$ is $O(\log N
+ \tau)$ with probability at least $1-e^{-\tau}$; this is known to be tight
up to constants, so the $\log N$ term cannot generally be removed.
This fact has been noted by previous works on matrix tail
inequalities~\citep[\emph{e.g.},][]{Tropp11a}, which also use this example
as an extreme case.
We note, however, that these previous works are not applicable to the case
of a countably infinite number of mean-zero Gaussian random variables $Z_i
\sim \mathcal{N}(0,\sigma_i^2)$ (or more generally, subgaussian random
variables), whereas the above inequality can be applied as long as the sum
of the $\sigma_i^2$ is finite.
\end{example}

\subsubsection{Principal component analysis}

Our next two examples uses Theorem~\ref{theorem:bernstein} to give spectral norm
error bounds for estimating the second moment matrix of a random vector
from i.i.d.~copies.
This is relevant in the context of (kernel) principal component
analysis of high (or infinite) dimensional
data~\citep[\emph{e.g.},][]{SSM99}.
\begin{example} \label{example:pca}
Let $x_1,\dotsc,x_n$ be i.i.d.~random vectors with $\Sig :=
\E[x_ix_i^\top]$, $K := \E[x_ix_i^\top x_ix_i^\top]$, and $\|x_i\|_2 \leq
\bar\ell$ almost surely for some $\bar\ell > 0$.
Let $X_i := x_ix_i^\top - \Sig$ and $\hat\Sig_n := n^{-1} \sum_{i=1}^n
x_ix_i^\top$.
We have $\lambda_{\max}(X_i) \leq \bar\ell^2 - \lambda_{\min}(\Sig)$.
Also, $\lambda_{\max}(n^{-1} \sum_{i=1}^n \E[X_i^2]) =
\lambda_{\max}(K-\Sig^2)$ and $\E[\tr(n^{-1} \sum_{i=1}^n \E[X_i^2])] =
\tr(K-\Sig^2)$.
By Theorem~\ref{theorem:bernstein},
\begin{equation*}
\Pr\left[ \lambda_{\max}\bigl(\hat\Sig_n - \Sig\bigr)
> \sqrt{\frac{2\lambda_{\max}(K-\Sig^2)t}{n}}
+ \frac{(\bar\ell^2 - \lambda_{\min}(\Sig))t}{3n} \right]
\leq \frac{\tr(K-\Sig^2)}{\lambda_{\max}(K-\Sig^2)} \cdot t(e^t-t-1)^{-1}
.
\end{equation*}
Since $\lambda_{\max}(-X_i) \leq \lambda_{\max}(\Sig)$, we also have
\begin{equation*}
\Pr\left[ \lambda_{\max}\bigl(\Sig - \hat\Sig_n\bigr)
> \sqrt{\frac{2\lambda_{\max}(K-\Sig^2)t}{n}}
+ \frac{\lambda_{\max}(\Sig)t}{3n} \right]
\leq \frac{\tr(K-\Sig^2)}{\lambda_{\max}(K-\Sig^2)} \cdot t(e^t-t-1)^{-1}
.
\end{equation*}
Therefore
\begin{equation*}
\Pr\left[ \bigl\|\hat\Sig_n - \Sig\bigr\|_2
> \sqrt{\frac{2\lambda_{\max}(K-\Sig^2)t}{n}}
+ \frac{\max\{ \bar\ell^2-\lambda_{\min}(\Sig),\ \lambda_{\max}(\Sig)\}
t}{3n} \right]
\leq \frac{\tr(K-\Sig^2)}{\lambda_{\max}(K-\Sig^2)} \cdot 2t(e^t-t-1)^{-1}
.
\end{equation*}
A similar result was given by~\citet[Lemma 1]{ZwaBla06} but for Frobenius
norm error rather than spectral norm error.
This is generally incomparable to our result, although spectral norm error
may be more appropriate in cases where the spectrum is slow to decay.
\end{example}

We now show that combining the bound from the previous example with sharper
dimension-dependent tail inequalities can sometimes lead to stronger
results.
\begin{example} \label{example:pca2}
Let $x_1,\dotsc,x_n$ be i.i.d.~random vectors with $\Sig :=
\E[x_ix_i^\top]$;
let $X_i := x_ix_i^\top - \Sig$ and $\hat\Sig_n := n^{-1} \sum_{i=1}^n
x_ix_i^\top$.
For any positive integer $d \leq \rank(\Sig)$, let $\Pi_{d,0}$ be the
orthogonal projector to the $d$-dimensional eigenspace of $\Sig$
corresponding to its $d$ largest eigenvalues, and let $\Pi_{d,1} := I -
\Pi_{d,0}$.
We have
\begin{align*}
\bigl\|\hat\Sig_n - \Sig\bigr\|_2
& \leq
\bigl\|\Pi_{d,0}(\hat\Sig_n - \Sig\bigr)\Pi_{d,0}\|_2
+ 2\bigl\|\Pi_{d,0}(\hat\Sig_n - \Sig\bigr)\Pi_{d,1}\|_2
+ \bigl\|\Pi_{d,1}(\hat\Sig_n - \Sig\bigr)\Pi_{d,1}\|_2
\\
& \leq
2\bigl\|\Pi_{d,0}(\hat\Sig_n - \Sig\bigr)\Pi_{d,0}\|_2
+ 2\bigl\|\Pi_{d,1}(\hat\Sig_n - \Sig\bigr)\Pi_{d,1}\|_2
.
\end{align*}
We can use the tail inequalities from this work to control
$\|\Pi_{d,1}(\hat\Sig_n - \Sig)\Pi_{d,1}\|_2$, and use potentially sharper
dimension-dependent inequalities to control $\|\Pi_{d,0}(\hat\Sig_n -
\Sig)\Pi_{d,0}\|_2$.

Let $\Sig_{d,0} := \Pi_{d,0} \Sig \Pi_{d,0}$, $\Sig_{d,1} := \Pi_{d,1} \Sig
\Pi_{d,1}$, $K_{d,1} := \E[(\Pi_{d,1} x_i x_i^\top \Pi_{d,1})^2]$, and
assume $\|\Pi_{d,1} x_i\|_2 \leq \bar\ell_{d,1}$ for all $i=1,\dotsc,n$
almost surely.
Furthermore, suppose there exists $\gamma_{d,0} > 0$ such that for all $i =
1,\dotsc,n$ and all vectors $\alpha$,
\begin{equation*}
\E\Bigl[ \exp\Bigl(\alpha^\top \Sig_{d,0}^{-1/2} x_i\Bigr) \Bigr] \leq
\exp\bigl(\gamma_{d,0} \|\alpha\|_2^2/2\bigr)
\end{equation*}
where $\Sig_{d,0}^{-1/2}$ is the matrix square-root of the Moore-Penrose
pseudoinverse of $\Sig_{d,0}$.
This condition states that every projection of $\Sig_{d,0}^{-1/2} x_i$ has
subgaussian tails.
In this case, the tail behavior of
$\|\Pi_{d,0}(\hat\Sig_n-\Sig)\Pi_{d,0}\|_2$ should not depend on the
dimensionality $d$.
Indeed, a covering number argument gives
\begin{equation*}
\Pr\left[ \bigl\|\Pi_{d,0}\bigl(\hat\Sig_n - \Sig\bigr)\Pi_{d,0}\bigr\|_2
> 2\gamma_{d,0} \|\Sig\|_2 \biggl( \sqrt{\frac{71 d + 16t}{n}} + \frac{5d +
t}{n} \biggr) \right] \leq 2e^{-t/2}
\end{equation*}
for any $t > 0$ (see Lemma~\ref{lemma:covariance-subgaussian} in
Appendix~\ref{appendix:outer}).
Combining this with the tail inequality from Example~\ref{example:pca}, we
have (for $t \geq 2.6$)
\begin{multline} \label{eq:pca}
\Pr\Biggl[
\bigl\|\hat\Sig_n - \Sig\bigr\|_2
>
4\gamma_{d,0} \|\Sig\|_2 \biggl( \sqrt{\frac{71 d + 16t}{n}} +
\frac{5d + t}{n} \biggr)
\\
{}+
2\sqrt{\frac{2\lambda_{\max}(K_{d,1}-\Sig_{d,1}^2)
\bigl(\log\bigl(\frac{\tr(K_{d,1} - \Sig_{d,1}^2)}{\lambda_{\max}(K_{d,1} -
\Sig_{d,1}^2)}\bigr) + t\bigr)}{n}} \\
{}+
\frac{2\max\{ \bar\ell_{d,1}^2-\lambda_{\min}(\Sig_{d,1}),\
\lambda_{\max}(\Sig_{d,1})\} \bigl(\log\bigl(\frac{\tr(K_{d,1} -
\Sig_{d,1}^2)}{\lambda_{\max}(K_{d,1} - \Sig_{d,1}^2)}\bigr) + t \bigr)}{3n}
\Biggr]
\leq 4e^{-t/2}
.
\end{multline}
\end{example}

\noindent {\bf Comparisons.}
We consider the following stylized scenario to compare the bounds from
Example~\ref{example:pca} and Example~\ref{example:pca2}.
\begin{enumerate}
\item The largest $d$ eigenvalues of $\Sig$ are all equal to $\|\Sig\|_2$,
and the remaining eigenvalues are smaller and rapidly decaying so
$\tr(\Sig_{d,1}) / \|\Sig\|_2$ is small.

\item $\bar\ell^2$ and $\bar\ell_{d,1}^2$ are within constant
factors of $\tr(\Sig)$ and $\tr(\Sig_{d,1})$, respectively; this simply
requires that the squared length of any $x_i$ never be more than a constant
factor times its expected squared length.

\item $\lambda_{\max}(K-\Sig^2)$ and $\lambda_{\max}(K_{d,1} -
\Sig_{d,1}^2)$ are within constant factors of $\lambda_{\max}(\Sig)^2$ and
$\lambda_{\max}(\Sig_{d,1})^2$, respectively; this is similar to the
previous condition.

\end{enumerate}
We will also ignore constant and logarithmic factors, as well as the
$\gamma_{d,0}$ factors.
The bound on $\|\hat\Sig_n\|_2$ from Example~\ref{example:pca2} then becomes
(roughly)
\begin{equation} \label{eq:pca-simple}
\|\Sig\|_2 \left( 1 + \sqrt{\frac{d}{n}} \right)
+ \|\Sig\|_2 \left( \sqrt{\frac{t}{n}} + \frac{t}{n}
+ \frac{(\tr(\Sig_{d,1})/\|\Sig\|_2) t}{n}
\right)
\end{equation}
whereas the bound from Example~\ref{example:pca} is
\begin{equation} \label{eq:pca-old-simple}
\|\Sig\|_2
+ \|\Sig\|_2 \left( \sqrt{\frac{t}{n}}
+ \frac{\Bigl(d + (\tr(\Sig_{d,1})/\|\Sig\|_2) \Bigr) t}{n}
\right)
.
\end{equation}
The main difference between these bounds is that the deviation term $t$
does not scale with $d$ in~\eqref{eq:pca-simple}, but it does
in~\eqref{eq:pca-old-simple}, so the exponential tail in the latter is much
weaker, as discussed in Subsection~\ref{section:discussion}.

We can also compare the bound from Example~\ref{example:pca2} to the case where
the $x_i$ are i.i.d.~Gaussian random vectors with mean zero and covariance
$\Sig$.
Arrange the $x_i$ as columns in a matrix $\hat{A}_n = [x_1|\dotsb|x_n]$, so
\[ \|\hat\Sig_n\|_2
= \frac1n \|\hat{A}_n\hat{A}_n^\top\|_2
= \frac1n \|\hat{A}_n\|_2^2
.
\]
Note that $\hat{A}_n$ has the same distribution as $\Sig^{1/2}Z$, where $Z$
is a matrix of independent standard Gaussian random variables.
The function $Z \mapsto \|\Sig^{1/2}Z\|_2 = \|\hat{A}_n\|_2$ is
$\|\Sig^{1/2}\|_2$-Lipschitz in $Z$, so by Gaussian
concentration~\citep{Pisier89},
\[
\Pr\left[ \|\hat{A}_n\|_2 > \E\bigl[\|\hat{A}_n\|_2\bigr] +
\sqrt{2\|\Sig\|_2t} \right] \leq e^{-t}
.
\]
The expectation can be bounded using a result of~\citet{Gordon85,Gordon88}:
\[
\E\bigl[\|\hat{A}_n\|_2\bigr]
= \E\bigl[\|\Sig^{1/2}Z\|_2\bigr]
\leq \|\Sig^{1/2}\|_2 \sqrt{n} + \|\Sig^{1/2}\|_F
.
\]
Putting these together, we obtain
\begin{equation*}
\Pr\Biggl[
\bigl\|\hat\Sig_n\bigr\|_2
>
\bigl\|\Sig\bigr\|_2
+ 2 \sqrt{\frac{\|\Sig\|_2\tr(\Sigma)}{n}} + 2\sqrt{\frac{2\|\Sig\|_2^2t}{n}}
+ \frac{\tr(\Sig) + 2\sqrt{2\tr(\Sig)\|\Sig\|_2t} + 2\|\Sig\|_2 t}{n}
\Biggr]
\leq e^{-t}
.
\end{equation*}
In our stylized scenario, this roughly implies a bound on
$\|\hat\Sig_n\|_2$ of the form
\begin{equation} \label{eq:pca-gaussian-simple}
\|\Sig\|_2 \left( 1 + \sqrt{\frac{d + \tr(\Sig_{d,1})/\|\Sig\|_2}{n}} +
\frac{d + \tr(\Sig_{d,1})/\|\Sig\|_2}{n} \right)
+ \|\Sig\|_2 \left( \sqrt{\frac{t}{n}} + \frac{t}{n} \right)
\end{equation}
Compared to~\eqref{eq:pca-simple}, we see that the main difference is
that $t$ does not scale with $\tr(\Sig_{d,1})/\|\Sig\|_2$
in~\eqref{eq:pca-gaussian-simple}, but it does in~\eqref{eq:pca-simple}.
Therefore the bounds are comparable (up to constant and logarithmic
factors) when the eigenspectrum of $\Sig$ is rapidly decaying after the
first $d$ eigenvalues.

\subsubsection{Approximate matrix multiplication}

Finally, we give an example about approximating a matrix product $AB^\top$
using non-uniform sampling of the columns of $A$ and $B$.
\begin{example} \label{example:sampling}
Let $A := [a_1|\dotsb|a_m]$ and $B := [b_1|\dotsb|b_m]$ be fixed matrices,
each with $m$ columns.
Assume $a_i \neq 0$ and $b_i \neq 0$ for all $i=1,\dotsc,m$.
If $m$ is very large, then the straightforward computation of the product
$AB^\top$ can be prohibitive.
An alternative is to take a small (non-uniform) random sample of the
columns of $A$ and $B$, say $a_{i_1}, b_{i_1}, \dotsc, a_{i_n}, b_{i_n}$,
and then compute a weighted sum of outer products
\begin{equation*}
\frac1n \sum_{j=1}^n \frac{a_{i_j}b_{i_j}^\top}{p_{i_j}}
\end{equation*}
where $p_{i_j} > 0$ is the \emph{a priori} probability of choosing the
column index $i_j \in \{1,\dotsc,m\}$ (the actual values of the
probabilities $p_i$ for $i=1,\dotsc,m$ are given below).
An analysis of this scheme was given by~\citet{MagZou11} with the stronger
requirement that the number of columns sampled be polynomially related to
the allowed failure probability.
Here we give an analysis in which the number of columns sampled depends
only logarithmically on the failure probability.

Let $X_1,\dotsc,X_n$ be i.i.d.~random matrices with the discrete
distribution given by
\begin{equation*}
\Pr\left[ X_j = \frac1{p_i} \begin{bmatrix} 0 & a_ib_i^\top \\
b_ia_i^\top & 0 \end{bmatrix} \right] = p_i
\propto \|a_i\|_2 \|b_i\|_2
\end{equation*}
for all $i = 1,\dotsc,m$, where $p_i := \|a_i\|_2 \|b_i\|_2 / Z$ and $Z :=
\sum_{i=1}^m \|a_i\|_2 \|b_i\|_2$.
Let
\begin{equation*}
\hat{M}_n := \frac1n \sum_{j=1}^n X_j
\quad \text{and} \quad
M := \begin{bmatrix} 0 & AB^\top \\ BA^\top & 0 \end{bmatrix}
.
\end{equation*}
Note that $\|\hat{M}_n - M\|_2$ is the spectral norm error of approximating
$AB^\top$ using the average of $n$ outer products $\sum_{j=1}^n a_{i_j}
b_{i_j}^\top / p_{i_j}$, where the indices are such that $i_j = i
\Leftrightarrow X_j = a_ib_i^\top / p_i$ for $j = 1,\dotsc,n$.

We have the following identities:
\begin{align*}
\E[X_j]
& = \sum_{i=1}^m p_i \left( \frac1{p_i} \begin{bmatrix} 0 & a_ib_i^\top
\\ b_ia_i^\top & 0 \end{bmatrix} \right)
= \begin{bmatrix} 0 & \sum_{i=1}^m a_ib_i^\top \\ \sum_{i=1}^m b_ia_i^\top
& 0 \end{bmatrix}
= M
\\
\tr(\E[X_j^2])
& = \tr\left( \sum_{i=1}^m p_i \left( \frac1{p_i^2} \begin{bmatrix}
a_ib_i^\top b_i a_i^\top & 0 \\ 0 & b_ia_i^\top a_ib_i^\top \end{bmatrix}
\right) \right)
= \sum_{i=1}^m \frac{2\|a_i\|_2^2\|b_i\|_2^2}{p_i}
= 2Z^2
\\
\tr(\E[X_j]^2)
& = \tr\left( \begin{bmatrix} AB^\top BA^\top & 0 \\ 0 & BA^\top AB^\top
\end{bmatrix} \right)
= 2\tr(A^\top A B^\top B)
;
\end{align*}
and the following inequalities:
\begin{align*}
\|X_j\|_2
& \leq \max_{i=1,\dotsc,m} \frac1{p_i}
\left\| \begin{bmatrix} 0 & a_ib_i^\top \\ b_ia_i^\top & 0 \end{bmatrix}
\right\|_2
= \max_{i=1,\dotsc,m} \frac{\|a_ib_i^\top\|_2}{p_i}
= Z
\\
\|\E[X_j]\|_2
& = \|AB^\top\|_2 \leq \|A\|_2 \|B\|_2
\\
\|\E[X_j^2]\|_2
& \leq \|A\|_2\|B\|_2 Z
.
\end{align*}
This means $\|X_j - M\|_2 \leq Z + \|A\|_2\|B\|_2$ and $\|\E[(X_j -
M)^2]\|_2 \leq \|\E[X_j^2] - M^2\|_2 \leq \|A\|_2\|B\|_2(Z +
\|A\|_2\|B\|_2)$, so Theorem~\ref{theorem:bernstein} and a union bound imply
\begin{multline*}
\Pr\left[
\bigl\|\hat{M}_n - M\bigr\|_2
> \sqrt{\frac{2\left(\|A\|_2\|B\|_2(Z + \|A\|_2\|B\|_2) \right)t}{n}} +
\frac{(Z + \|A\|_2\|B\|_2)t}{3n}
\right]
\\
\leq 4\left(\frac{Z^2 - \tr(A^\top A B^\top B)}{\|A\|_2\|B\|_2(Z +
\|A\|_2\|B\|_2)} \right) \cdot t(e^t-t-1)^{-1}
.
\end{multline*}
Let $r_A := \|A\|_F^2 / \|A\|_2^2 \in [1,\rank(A)]$ and $r_B := \|B\|_F^2 /
\|B\|_2^2 \in [1,\rank(B)]$ be the numerical (or stable) rank of $A$ and
$B$, respectively.
Since $Z / (\|A\|_2\|B\|_2) \leq \|A\|_F \|B\|_F / (\|A\|_2\|B\|_2) =
\sqrt{r_Ar_B}$, we have the simplified (but slightly looser) bound
\begin{equation*}
\Pr\left[
\frac{\bigl\|\hat{M}_n - M\bigr\|_2}{\|A\|_2\|B\|_2}
>
2\sqrt{\frac{(1 + \sqrt{r_Ar_B})(\log(4\sqrt{r_Ar_B}) + t)}{n}}
+
\frac{2(1 + \sqrt{r_Ar_B})(\log(4\sqrt{r_Ar_B}) + t)}{3n}
\right]
\leq e^{-t}
.
\end{equation*}
Therefore, for any $\epsilon \in (0,1)$ and $\delta \in (0,1)$, if
\begin{equation*}
n \geq \left( \frac83 + 2\sqrt{\frac53} \right)
\frac{(1 + \sqrt{r_Ar_B})(\log(4\sqrt{r_Ar_B}) +
\log(1/\delta))}{\epsilon^2}
,
\end{equation*}
then with probability at least $1-\delta$ over the random choice of column
indices $i_1,\dotsc,i_n$,
\begin{equation*}
\left\| \frac1n \sum_{j=1}^n \frac{a_{i_j}b_{i_j}^\top}{p_{i_j}} - AB^\top
\right\|_2 \leq \epsilon \|A\|_2 \|B\|_2
.
\end{equation*}
\end{example}

\subsubsection*{Acknowledgements}

We are grateful to Alex Gittens for useful comments and pointing out a
subtle mistake in our proof of Theorem~\ref{theorem:laplace-prob} in an
earlier draft, and to Joel Tropp for his many comments and suggestions.

\subsection*{References}
{\def\section*#1{}%
\bibliography{concentration}%
\bibliographystyle{plainnat}}

\appendix

\section{Sums of random vector outer products}
\label{appendix:outer}

The following lemma is a tail inequality for smallest and largest
eigenvalues of the empirical covariance matrix of subgaussian random
vectors.
This result (with non-explicit constants) was originally obtained
by~\citet{LPRTJ05}.
\begin{lemma}
\label{lemma:covariance-subgaussian}
Let $x_1,\dotsc,x_n$ be random vectors in $\R^d$ such that, for some
$\gamma \geq 0$,
\begin{gather*}
\E\left[ x_ix_i^\top \ \Big| \ x_1,\dotsc,x_{i-1} \right] = I
\quad \text{and} \\
\E\left[ \exp\left( \alpha^\top x_i \right) \ \Big| \ x_1,\dotsc,x_{i-1}
\right] \leq \exp\left( \|\alpha\|_2^2 \gamma / 2 \right)
\quad \text{for all $\alpha \in \R^d$}
\end{gather*}
for all $i=1,\dotsc,n$, almost surely.
For all $\epsilon_0 \in (0,1/2)$ and $\delta \in (0,1)$,
\[
\Pr\Biggl[\
\lambda_{\max}\left( \frac1n \sum_{i=1}^n x_ix_i^\top \right)
 > 1 + \frac1{1-2\epsilon_0} \cdot \epsilon_{\epsilon_0,\delta,n}
\quad \text{or} \quad
\lambda_{\min}\left( \frac1n \sum_{i=1}^n x_ix_i^\top \right)
< 1 - \frac1{1-2\epsilon_0} \cdot \epsilon_{\epsilon_0,\delta,n}
\ \Biggr] \leq \delta
\]
where
\[
\epsilon_{\epsilon_0,\delta,n} :=
\gamma \cdot
\left(
\sqrt{\frac{32\left( d\log(1+2/\epsilon_0) + \log(2/\delta)
\right)}{n}}
+ \frac{2\left( d\log(1+2/\epsilon_0) + \log(2/\delta)
\right)}{n}
\right)
.
\]
\end{lemma}
\begin{remark}
In our applications of this lemma, we will simply choose $\epsilon_0 :=
1/4$ for concreteness.
\end{remark}

We give the proof of Lemma~\ref{lemma:covariance-subgaussian} for
completeness.

The subgaussian property most readily lends itself to bounds on linear
combinations of subgaussian random variables.
However, we are interested in bounding certain quadratic combinations.
Therefore we bootstrap from the bound for linear combinations to bound the
moment generating function of the quadratic combinations; from there, we
can obtain the desired tail inequality.

The following lemma relates the moment generating function to a tail
inequality.
\begin{lemma} \label{lemma:mgf-bound}
Let $W$ be a non-negative random variable.
For any $\eta \in \R$,
\[ \E\left[ \exp\left( \eta W \right) \right]
- \eta \E\left[ W \right]
- 1
= \eta \int_0^\infty \left( \exp\left( \eta t \right) - 1 \right)
\cdot \Pr\left[ W > t \right] \cdot dt
.
\]
\end{lemma}
\begin{proof}
Integration-by-parts.
\end{proof}

The next lemma gives a tail inequality for any particular Rayleigh quotient
of the empirical covariance matrix.
\begin{lemma} \label{lemma:rayleigh}
Let $x_1,\dotsc,x_n$ be random vectors in $\R^d$ such that, for some
$\gamma \geq 0$,
\begin{gather*}
\E\left[ x_ix_i^\top \ \Big| \ x_1,\dotsc,x_{i-1} \right] = I
\quad \text{and} \\
\E\left[ \exp\left( \alpha^\top x_i \right) \ \Big| \ x_1,\dotsc,x_{i-1}
\right] \leq \exp\left( \|\alpha\|_2^2 \gamma / 2 \right)
\quad \text{for all $\alpha \in \R^d$}
\end{gather*}
for all $i=1,\dotsc,n$, almost surely.
For all $\alpha \in \R^d$ such that $\|\alpha\|_2 = 1$, and all $\delta \in
(0,1)$,
\[
\Pr\left[ \alpha^\top \left( \frac1n \sum_{i=1}^n x_ix_i^\top \right)
\alpha > 1 + \sqrt{\frac{32\gamma^2 \log(1/\delta)}{n}} + \frac{2\gamma
\log(1/\delta)}{n}
\right] \leq \delta
\]
and
\[
\Pr\left[ \alpha^\top \left( \frac1n \sum_{i=1}^n x_ix_i^\top \right)
\alpha < 1 - \sqrt{\frac{32\gamma^2 \log(1/\delta)}{n}}
\right] \leq \delta
. \]
\end{lemma}
\begin{proof}
Fix $\alpha \in \R^d$ with $\|\alpha\|_2 = 1$.
For $i = 1,\dotsc,n$, let $W_i := (\alpha^\top x_i)^2$, so $\E[W_i] = 1$.
For any $t \geq 0$, using Chernoff's bounding method gives
\begin{align*}
\lefteqn{
\E\left[ \I\left[ W_i > t \right] \ | \ x_1,\dotsc,x_{i-1} \right]
}
\\
& \leq \inf_{\eta>0} \left\{ \E\left[ \I\left[ \exp\left( \eta
|\alpha^\top x_i| \right) > e^{\eta\sqrt{t}} \right] \ \Big| \
x_1,\dotsc,x_{i-1}\right] \right\} \\
& \leq \inf_{\eta>0} \left\{
e^{-\eta\sqrt{t}}
\cdot \left(
\E\left[ \exp\left( \eta \alpha^\top x_i \right) \ \Big| \
x_1,\dotsc,x_{i-1} \right]
+ \E\left[ \exp\left( -\eta \alpha^\top x_i \right) \ \Big| \
x_1,\dotsc,x_{i-1} \right]
\right)
\right\} \\
& \leq \inf_{\eta>0} \left\{
2 \exp\left( - \eta \sqrt{t} + \eta^2 \gamma / 2 \right)
\right\} \\
& = 2 \exp\left( - \frac{t}{2\gamma} \right)
.
\end{align*}
So by Lemma~\ref{lemma:mgf-bound}, for any $\eta < 1/(2\gamma)$,
\begin{align*}
\E\left[ \exp\left( \eta W_i \right) \ | \ x_1,\dotsc,x_{i-1}\right]
& \leq 1 + \eta
+ \eta \int_0^\infty \left( \exp\left( \eta t \right) - 1 \right)
\cdot 2 \exp\left( - \frac{t}{2\gamma} \right) \cdot dt \\
& = 1 + \eta + \frac{8\eta^2\gamma^2}{1 - 2\eta\gamma} \\
& \leq \exp\left( \eta + \frac{8\eta^2\gamma^2}{1 -
2\eta\gamma} \right)
\end{align*}
and therefore
\[ \E\left[ \exp\left( \eta \sum_{i=1}^n W_i \right) \right]
\leq \exp\left( n\eta +
\frac{8n\eta^2\gamma^2}{1-2\eta\gamma} \right)
. \]
Using Chernoff's bounding method twice more gives
\begin{align*}
\Pr\left[ \sum_{i=1}^n W_i > n + t \right]
& \leq \inf_{0\leq\eta<1/(2\gamma)} \left\{ \exp\left( -t\eta +
\frac{8n\eta^2\gamma^2}{1-2\eta\gamma} \right) \right\} \\
& = \exp\left( -\frac{8n\gamma^2 + \gamma t - \sqrt{8n\gamma^2
\left( 8n\gamma^2 + 2\gamma t \right)}}{2\gamma^2} \right)
\end{align*}
and
\[
\Pr\left[ \sum_{i=1}^n W_i < n - t \right]
\leq \inf_{\eta\leq0} \left\{ \exp\left( t\eta +
\frac{8n\eta^2\gamma^2}{1-2\eta\gamma} \right) \right\} \\
\leq \exp\left( -\frac{t^2}{32n\gamma^2} \right)
.
\]
The claim follows.
\end{proof}

In order to bound the smallest and largest eigenvalues of the empirical
covariance matrix, we apply the bound for the Rayleigh quotient in
Lemma~\ref{lemma:rayleigh} together with a covering argument.
\begin{lemma}[\citealp{Pisier89}]
\label{lemma:pisier}
For any $\epsilon_0 > 0$, there exists $Q \subseteq \mathcal{S}^{d-1} := \{ \alpha \in
\R^d : \|\alpha\|_2 = 1 \}$ of cardinality $\leq (1+2/\epsilon_0)^d$ such that
$\forall \alpha \in \mathcal{S}^{d-1} \exists q \in Q \centerdot \|\alpha-q\|_2 \leq
\epsilon_0$.
\end{lemma}

\begin{proof}[Proof of Lemma~\ref{lemma:covariance-subgaussian}]
Let $\hat\Sig := (1/n) \sum_{i=1}^n x_ix_i^\top$, let $\mathcal{S}^{d-1} :=
\{ \alpha \in \R^d : \|\alpha\|_2 = 1 \}$ be the unit sphere in $\R^d$, and
let $Q \subset \mathcal{S}^{d-1}$ be an $\epsilon_0$-cover of
$\mathcal{S}^{d-1}$ of minimum size with respect to $\|\cdot\|_2$.
By Lemma~\ref{lemma:pisier}, the cardinality of $Q$ is at most
$(1+2/\epsilon_0)^d$.
Let $E$ be the event
\[
\max \left\{ |q^\top (\hat\Sig - I) q| : q \in Q \right\} \leq
\epsilon_{\epsilon_0,\delta,n}
.
\]
By Lemma~\ref{lemma:rayleigh} and a union bound, $\Pr[E] \geq 1-\delta$.
Now assume the event $E$ holds.
Let $\alpha_0 \in \mathcal{S}^{d-1}$ be such that
$|\alpha_0^\top (\hat\Sig - I) \alpha_0| = \max\{ |\alpha^\top
(\hat\Sig - I) \alpha| : \alpha \in \mathcal{S}^{d-1} \} = \|\hat\Sig-I\|_2$.
Using the triangle and Cauchy-Schwarz inequalities, we have
\begin{align*}
\|\hat\Sig-I\|_2
= |\alpha_0^\top (\hat\Sig - I) \alpha_0|
& = \min_{q\in Q} |q^\top (\hat\Sig - I) q
+ \alpha_0^\top (\hat\Sig - I) \alpha_0 - q^\top (\hat\Sig - I) q| \\
& \leq \min_{q\in Q} |q^\top (\hat\Sig - I) q|
+ |\alpha_0^\top (\hat\Sig - I) \alpha_0 - q^\top (\hat\Sig - I) q| \\
& = \min_{q\in Q} |q^\top (\hat\Sig - I) q| + |\alpha_0^\top (\hat\Sig - I)
(\alpha_0-q) - (q-\alpha_0)^\top (\hat\Sig - I) q| \\
& \leq \min_{q\in Q} |q^\top (\hat\Sig - I) q| + \|\alpha_0\|_2 \|\hat\Sig
- I\|_2 \|\alpha_0-q\|_2 + \|q-\alpha_0\|_2 \|\hat\Sig - I\|_2 \|q\|_2 \\
& \leq \epsilon_{\epsilon_0,\delta,n} + 2\epsilon_0\|\hat\Sig - I\|_2
\end{align*}
so
\[
\max\left\{ \lambda_{\max}(\hat\Sig) - 1, \ 1 - \lambda_{\min}(\hat\Sig)
\right\}
= \|\hat\Sig-I\|_2
\leq \frac1{1-2\epsilon_0} \cdot \epsilon_{\epsilon_0,\delta,n}
.
\]
\end{proof}

\end{document}